\documentclass[11pt, letterpaper, oneside]{amsart}

\usepackage{latexsym, amsmath, amssymb, amsfonts, amscd,bm}
\usepackage{amsthm}
\usepackage{t1enc}
\usepackage[mathscr]{eucal}
\usepackage{indentfirst}
\usepackage{graphicx, pb-diagram}
\usepackage{fancyhdr}
\usepackage{fancybox}
\usepackage{enumerate}
\usepackage{color}
\usepackage[all]{xy}
\usepackage{hyperref}

\usepackage{url}

\theoremstyle{plain}
\newtheorem{thm}{Theorem}[section]

\newtheorem{prop}[thm]{Proposition}

\newtheorem{lemma}[thm]{Lemma}
\newtheorem{cor}[thm]{Corollary}

\theoremstyle{definition}
\newtheorem{defi}[thm]{Definition}

\theoremstyle{remark}
\newtheorem{remark}[thm]{Remark}

\newcommand{\RR}{\ensuremath{\mathbb R}}
\newcommand{\g}{\ensuremath{\mathfrak{g}}}

\newcommand{\li}{\ensuremath{L_{\infty}}}

\newcommand{\cB}{\mathcal{B}}

\newcommand{\cL}{\mathcal{L}}

\newcommand{\cC}{\mathcal{C}}

\definecolor{forest}{rgb}{0,0.5,0} 

\newcommand{\dw}{\ensuremath{d_{tot}}}
\newcommand{\vs}{\varsigma}

\newcommand{\ham}[1]{\Omega^{#1}_{\mathrm{Ham}}\left(M\right)}
\newcommand{\ip}[1]{\iota_{v_{#1}}}

\newcommand{\alphak}[1]{\alpha_{1} \otimes \cdots \otimes \alpha_{#1}}
\newcommand{\alphadk}[1]{\alpha_{1},\hdots,\alpha_{#1}}

\newcommand{\vk}[1]{v_{\alpha_{1}} \wedge \cdots \wedge  v_{\alpha_{#1}}}

\newcommand{\prim}{\varphi}

\begin{document}

\title{A cohomological framework\\ for homotopy moment maps}

\author{Ya\"el Fr\'egier}
\address{Laboratoire de Math\'ematiques de Lens (Universit\'e d'Artois) and Institut f\"ur Mathematik (Universit\"at Z\"urich)}
\email{yael.fregier@math.uzh.ch, yael.fregier@gmail.com}

\author{Camille Laurent-Gengoux}
\address{
Laboratoire et D\'epartement de Math\'ematiques
  UMR 7122
  Universit\'e de Metz et CNRS
  Bat. A, Ile du Saulcy
  F-57045 Metz Cedex 1, France}
\email{camille.laurent-gengoux@univ-lorraine.fr}

\author{Marco Zambon}
\address{Universidad Aut\'onoma de Madrid (Departamento de Matem\'aticas) and ICMAT(CSIC-UAM-UC3M-UCM),
Campus de Cantoblanco,
28049 - Madrid, Spain. 
Current address: KU Leuven, Department of Mathematics, Celestijnenlaan 200B box 2400, BE-3001 Leuven, Belgium}
\email{marco.zambon@uam.es, marco.zambon@icmat.es, marco.zambon@wis.kuleuven.be}


\begin{abstract}
Given a Lie group acting on a manifold $M$ preserving a closed $n+1$-form $\omega$, the notion of homotopy moment map for this action was introduced in \cite{FRZ}, in terms of $L_{\infty}$-algebra morphisms.
In this note we describe homotopy moment maps as  coboundaries of a certain complex. This description simplifies greatly computations, and we use it to study various properties of homotopy moment maps: their relation to equivariant cohomology, their obstruction theory,  how they induce new ones on mapping spaces, and their  equivalences. 
The results we obtain extend some of the results of \cite{FRZ}. 
\end{abstract}

\maketitle

\setcounter{tocdepth}{1} 
\tableofcontents

 \section*{Introduction}
Recall that a symplectic form is a 
closed, non-degenerate 2-form.
It is natural to consider symmetries of a given symplectic manifold, that is, a Lie group acting on a manifold, preserving the symplectic form. Among such actions, 
a nice subclass  is given by actions that admit a moment map; in that case  the infinitesimal generators of the action are hamiltonian vector fields. Actions admitting a moment map enjoy remarkable geometric, algebraic and topological properties, that have been studied extensively in the literature (e.g. symplectic reduction, the relation to equivariant cohomology and localization, convexity theorems,...)

In this note we consider closed $n+1$-forms for some $n\ge 1$. 
When they are  non-degenerate, they are called multisymplectic form,
and are higher analogues of symplectic forms which appear naturally in classical field theory.

Recently 
Rogers \cite{RogersL}
(see also \cite{HDirac}) showed that 
the algebraic structure underlying  a manifold with a closed $n+1$-form   $\omega$ is the one of an $L_{\infty}$-algebra. This allowed  \cite{FRZ} for a natural extension of the notion of moment map to 
closed forms of arbitrary degree, called \emph{homotopy moment map}. The latter is phrased in terms of $L_{\infty}$-algebra morphisms.\\

The first contribution of this note is to construct, out of the action of a Lie group $G$ on a manifold $M$, a chain complex $\cC$  with the following property: 
\begin{itemize}
\item any invariant closed form $\omega$ gives rise to a cocycle $\tilde{\omega}$ in $\cC$
\item homotopy moment maps are given exactly by the primitives of $\tilde{\omega}$.
\end{itemize}
The chain complex $\cC$ is simply the product of the Chevalley-Eilenberg complex of the Lie algebra of $G$, with the de Rham complex of $M$. The action is encoded by the cocycle $\tilde{\omega}$.
Notice that by the above the set of homotopy moment maps (for a fixed  $\omega$) has the structure of an affine space, which is unexpected since $L_{\infty}$-algebra morphisms are generally very non-linear objects.

This characterization of homotopy moment maps is very useful: $L_{\infty}$-algebra morphisms are usually quite intricate and cumbersome to work with in an explicit way, while working with coboundaries in a complex is much simpler. In this note we use the above characterization to:
\begin{itemize}
\item show that certain extensions of $\omega$ in the Cartan model give rise to homotopy moment maps (see \S \ref{eqcohom}),
\item give cohomological obstructions to the existence of homotopy moment maps (see \S \ref{obstr}),
\item show that a homotopy moment map for a $G$-action on $(M,\omega)$ induces one on   $Maps(\Sigma, M)$, the space of maps from 
any closed and oriented manifold $\Sigma$ into $M$, endowed with the closed  form obtained from $\omega$ by transgression (see \S \ref{LM}),
\item obtain a natural notion of equivalence of homotopy moment maps, both under the requirement that $\omega$ be kept fixed and allow $\omega$ to vary (see \S \ref{equiv}). 
We show that it is compatible with the geometric notion of equivalence induced by isotopies of the manifold $M$,  and with the notion of equivalence of $L_{\infty}$-morphisms (see Appendix \ref{sec:dolg}). 
\end{itemize}

In \S \ref{eqcohom} and \S \ref{obstr} we obtain results similar to those of \cite{FRZ}, but with much less computational effort. 
The results obtained in \S \ref{equiv} are a significant extension of  results obtained in \cite{FRZ}, where only closed 3-forms and loop spaces were considered. The equivalences introduced in \ref{equiv} and their properties extend and justify 
the work carried out for closed 3-forms
in \cite[\S 7.4]{FRZ}.

One more application of the  characterization of moment maps as coboundaries in $\cC$ is the following. 
Given  two manifolds endowed with closed forms, 
their cartesian product $(M_1\times M_2,\omega_1\wedge \omega_2)$ is again an object of the same kind. This construction restricts to the multisymplectic category, but not to the symplectic one. 
The above characterization of moment maps is used in 
\cite{ProductMomaps} to construct homotopy moment maps for cartesian products.
\bigskip

\noindent\textbf{Remark:}
Recall that if $\mathfrak{X}$ is a Lie algebra, a \emph{$\mathfrak{X}$-differential algebra} \cite[\S 3]{MeinEqCoho} is a graded commutative algebra $\Omega=\oplus_{i\ge 0}\Omega^i$ with graded derivations $\iota_v, \cL_v$ of degrees $-1,0$ (depending linearly on $v\in \mathfrak{X}$) and a derivation $d$ of degree $1$ such that the Cartan relations hold:
\begin{align*}
[d,d]&=0\;\;\;[\cL_v,d]=0,\;\;\;[\iota_v,d]=\cL_v\\
[\iota_v,\iota_w]&=0,\;\;\;[\cL_v,\cL_w]=\cL_{[v,w]_{\mathfrak{X}}},\;\;[\cL_v,\iota_w]=\iota_{[v,w]_{\mathfrak{X}}}.
\end{align*}
This note is written in terms of geometric objects,
but most of it  applies also to the algebraic setting obtained
   replacing
   the setting we assume in \S \ref{momap} with:
 \begin{center}
\parbox[c]{12.6cm}{\begin{center}
$\mathfrak{X}$ a  Lie algebra, $\Omega$ a  $\mathfrak{X}$-differential algebra, $\omega\in \Omega^{n+1}$ with $d\omega=0$.\\ 
$\g$  a Lie algebra and $\rho\colon \g \to \mathfrak{X}$ a Lie algebra morphism, so that $\cL_{\rho(x)}\omega=0$ for all $x\in \g$.\end{center}
}
 \end{center}

\noindent\textbf{Remark:} 
The existence and uniqueness of homotopy moment maps is also studied by
Ryvkin and Wurzbacher in \cite{WurzRyvkinMomaps}, where the authors obtain independently results similar to ours on this subject, putting an emphasis on the differential geometry of multisymplectic forms.

\noindent\textbf{Acknowledgements:} Y.F. acknowledges support from Marie Curie Grant IOF-hqsmcf-274032 and from the Max Planck Institute for Mathematics. This research was partly supported by the NCCR SwissMAP, funded by the Swiss National Science Foundation (Grant No. 200020\_149150/1).

 M.Z. thanks Camilo Arias Abad, Camilo Rengifo,
Chris Rogers and Bernardo Uribe for useful discussions. M.Z. acknowledges partial financial support by grants  
MTM2011-22612 and ICMAT Severo Ochoa  SEV-2011-0087 (Spain), and
Pesquisador Visitante Especial grant  88881.030367/2013-01 (CAPES/Brazil).

\section{Closed forms}

We recall briefly how some notions from symplectic geometry apply to closed differential forms of arbitrary degree.  

\begin{defi} \label{hamiltonian}
Let $(M,\omega)$ be a {\bf pre-$n$-plectic} manifold, i.e., $M$ is a  manifold and $\omega$ a closed $n+1$-form.  An $(n-1)$-form $\alpha$
is {\bf Hamiltonian} iff there exists a vector field $v_\alpha \in \mathfrak{X}(M)$ such that
\[
d\alpha= -\ip{\alpha} \omega.
\]
We say $v_\alpha$ is a {\bf Hamiltonian vector field} for $\alpha$. 
The set of Hamiltonian $(n-1)$-forms 
is denoted
as $\ham{n-1}$.
\end{defi}

In analogy to symplectic geometry, one can endow the set of Hamiltonian $(n-1)$-forms with a skew-symmetric bracket, which however is not a Lie bracket. If one passes from $\ham{n-1}$ to a larger space, one obtains
an $L_{\infty}$-algebra \cite{LadaStasheff}, which was constructed  essentially in \cite[Thm. 5.2]{RogersL}, and  generalized slightly in \cite[Thm. 6.7]{HDirac}.
\begin{defi}  
Given a  pre-$n$-plectic manifold $(M,\omega)$, the \textbf{observables} form  an $L_{\infty}$ algebra, denoted
$L_{\infty}(M,\omega):=(L,\{l_{k} \})$. The underlying graded vector space  is given by 
\[
L_{i} =
\begin{cases}
\ham{n-1} & i=0,\\
\Omega^{n-1+i}(M) & -n+1 \leq i < 0.
\end{cases}
\]
The maps  $\left \{l_{k} \colon L^{\otimes k} \to L| 1
  \leq k < \infty \right\}$ are defined as
\[ 
l_{1}(\alpha)=d\alpha,
\]
if $\deg({\alpha})>0$, and for all $k>1$ 
\[
l_{k}(\alphadk{k}) =
\begin{cases}
0 & \text{if $\deg({\alphak{k}}) < 0$}, \\
\vs(k) \iota(\vk{k}) \omega  & \text{if
  $\deg({\alphak{k}})=0$}, 
  \end{cases}
\]
 where $v_{\alpha_{i}}$ is any Hamiltonian vector field
associated to $\alpha_{i} \in \ham{n-1}$. Here\footnote{So $\varsigma(k)=1,1,-1,-1,1,\dots$ for $k=1,2,3,4,5,\dots$.} 
 $\varsigma(k)=-(-1)^{k(k+1)/2}$. Notice that $\varsigma(k-1)\varsigma(k)=(-1)^k$ for all $k$.
\end{defi}

Among the Lie group actions on $M$ that preserve $\omega$, it is natural to consider those whose infinitesimal generators are hamiltonian vector fields. This leads to the following notion \cite[Def. 5.1]{FRZ}.

 \begin{defi}\label{def:momap}
A \textbf{(homotopy) moment map} for the action of $G$ on $(M,\omega)$ is a $\li$-morphism $f\colon \g \to L_{\infty}(M,\omega)$ such  
that for all $x\in \g$
\begin{equation}\label{eq:mom}
d(f_1(x))=-\iota(v_x)\omega.
\end{equation}
 
Saying that $f$ is a $\li$-morphism means that it consists of components  $f_k\colon \wedge^k \g\to \Omega^{n-k}(M)$
(for $k=1,\dots,n$)
satisfying

\begin{multline} \label{main_eq_1}
\sum_{1 \leq i < j \leq k}
(-1)^{i+j+1}f_{k-1}([x_{i},x_{j}],x_{1},\ldots,\widehat{x_{i}},\ldots,\widehat{x_{j}},\ldots,x_{k})\\
=df_{k}(x_{1},\ldots,x_{k}) + \vs(k)\iota(v_{x_1}\wedge \cdots \wedge v_{x_k})\omega
\end{multline}
 for $2 \leq k \leq n$, as well as  
\begin{multline} \label{main_eq_2}
\sum_{1 \leq i < j \leq n+1}
(-1)^{i+j+1}f_{n}([x_{i},x_{j}],x_{1},\ldots,\widehat{x_{i}},\ldots,\widehat{x_{j}},\ldots,x_{n+1})
=\vs(n+1)\iota(v_{x_1}\wedge \cdots \wedge v_{x_{n+1}})\omega.
\end{multline}
 \end{defi}

\section{A double complex  encoding moment maps}\label{momap}
The set-up in the whole of this note is the following:

\begin{center}
\fbox{
\parbox[c]{12.6cm}{\begin{center}
$(M,\omega)$ is a   pre-$n$-plectic manifold,\\
$G$ is a Lie group acting on $M$ preserving $\omega$.
\end{center}
}}
 \end{center}
 We denote the Lie algebra of $G$ by $\g$, elements of $\g$ by $x$, and the corresponding
 infinitesimal generators of the action (which are vector fields on $M$) by $v_{x}$.\\

In this section we introduce a complex with the property that suitable coboundaries correspond bijectively to moment maps for the action of $G$ on $(M,\omega)$.

The manifold $M$ and the Lie algebra $\g$ give rise to a  double complex
\begin{equation}\label{eq:dc}
(\wedge^{\ge1} \g^*\otimes \Omega(M), d_\g,d),
\end{equation}
where $d_\g$ is the Chevalley-Eilenberg differential of $\g$ and $d$ is the De Rham differential of $M$.  
We consider the total complex, which we denote by $\cC$,  
with differential $$\dw:=d_\g\otimes 1+1\otimes d.$$ We use the Koszul sign convention, hence,  on an element of $\wedge^k \g^*\otimes \Omega(M)$, $\dw$ acts as $d_\g + (-1)^kd$.
 
We first need a lemma, which appears (using a slightly different notation) as the Extended Cartan Formula in \cite[Lemma 3.4]{MadsenSwannClosed}, and which we present without\footnote{It can be proven   by a direct computation, extending the proof of \cite[Lemma  7.2]{FRZ}.} 
 proof.  
 \begin{lemma}\label{tech_lemma}
Let $M$ be a manifold 
and let  $\Omega$ be an  $N$-form (not necessarily closed).
 For all $k \geq 2$ and all vector fields $v_1,\dots,v_k$  we have:
\begin{align*}  
(-1)^{k}d \iota(v_{1} \wedge\cdots \wedge v_{k}) \Omega =& 
 \sum_{1 \leq i < j \leq
  k} (-1)^{i+j} \iota([v_{i},v_{j}] \wedge v_{1} \wedge \cdots
  \wedge \widehat{v}_{i} \wedge \cdots \wedge \widehat{v}_{j} \wedge \cdots \wedge v_{k})\Omega\\
&+\sum_{1=1}^{k} (-1)^{i} \iota( v_{1} \wedge \cdots
  \wedge \widehat{v}_{i} \wedge \cdots \wedge {v}_{k})\cL_{v_i}\Omega\\
&+ \iota( v_{1} \wedge \cdots
  \wedge   {v}_{k}) d\Omega. 
\end{align*}
\end{lemma}
 
{\begin{remark}
The Lie derivative of a form $\Omega$ along an  multivector field $V=v_{1} \wedge\cdots \wedge v_{k}$ is defined by $\cL_V\Omega:=d\iota(V)\Omega-(-1)^k\iota(V)d\Omega$ 
\cite[Def. A2]{ForgerPoisMulti}.  From the above we deduce that $(-1)^k\cL_V\Omega$ equals 
the first two terms on the right hand side of the identity in Lemma \ref{tech_lemma}.
\end{remark}
}
 {
\begin{lemma}\label{tilde} For any $G$-invariant $\sigma\in \Omega^N(M)$ define \begin{equation}\label{eq:omegak}
{\sigma}_k \colon \wedge^k\g \to \Omega^{N-k}(M),\;\; (x_1,\dots,x_k)\mapsto \iota{(v_{x_1}\wedge\dots\wedge v_{x_k})}\sigma.
\end{equation}
and
$\widetilde{\sigma}:=\sum_{k=1}^{N}(-1)^{k-1}{\sigma}_k$.  
The map $$\widetilde{}\;\; \colon (\Omega(M)^G,d)\to (\wedge^{\ge 1} \g^*\otimes \Omega(M),\dw),\;\; \sigma\mapsto \widetilde{\sigma}$$ intertwines the differentials, that is: $\dw \widetilde{\sigma}=\widetilde{d\sigma}
$.
\end{lemma}
\begin{proof}  Lemma \ref{tech_lemma} implies that
  $(-1)^{k}d\sigma_{k}=d_\g\sigma_{k-1}+(d\sigma)_k$ for all $k\ge 2$. 
Hence
\begin{align*}
\dw \widetilde{\sigma}&=
 \sum_{k=1}^{N}(-1)^{k-1}(d_\g\sigma_k +(-1)^k d\sigma_k)\\
&= \sum_{k=2}^{N+1}(-1)^{k}d_\g\sigma_{k-1} -\sum_{k=1}^{N}  d \sigma_k
\\
&= \sum_{k=2}^{N+1}((-1)^{k}d_\g\sigma_{k-1} - d \sigma_k)-d\sigma_1
=\widetilde{d\sigma},
\end{align*} 
where in the third equality we used $\sigma_{N+1}=0$ , and 
in the last one we used the above equation and $-d\sigma_1=(d\sigma)_1$ (the latter follows from $d(\iota_{v_x}\sigma)+
\iota_{v_x}d\sigma=\cL_{v_x}\sigma=0$). \end{proof}

Since $\omega$ is a closed differential form, from Lemma \ref{tilde} we obtain:

\begin{cor}\label{omegacl}  
 $\widetilde{\omega}$ is $\dw$-closed.
\end{cor}

The next proposition  states that moment maps  for the action of $G$ on $(M,\omega)$ correspond bijectively to primitives of $\widetilde{\omega}$ in $(\cC,\dw)$. In particular, moment maps form an affine space, which is somewhat surprising since generally $L_{\infty}$-morphisms are very non-linear objects.

\begin{prop}\label{lem:doublemomap}
Let $\prim=\prim_1+\dots+\prim_n$, with $\prim_k \in \wedge^k \g^*\otimes \Omega^{n-k}(M)$. Then:     $\dw\prim=\widetilde{\omega}$ if{f} $$f_k:=\varsigma(k)\prim_k \colon \wedge^k\g \to \Omega^{n-k}(M),$$ for $k=1,\dots,n$, are the components of a homotopy moment map for the action of $G$ on $(M,\omega)$.
\end{prop}
\begin{proof}
$\dw \prim=
\sum_{k=2}^{n+1} d_\g\prim_{k-1}+\sum_{k=1}^{n}(-1)^{k} d\prim_{k}$ is equal to $\widetilde{\omega}$ if{f}  we have
\begin{align}
\label{eq:third} -d\prim_1&=\omega_1\\
\label{eq:first} d_\g\prim_{k-1}+ (-1)^{k} d\prim_{k}&= (-1)^{k-1}\omega_k\;\;\;\;\;\text{for all $2\le k \le n$}\\
 \label{eq:second}d_\g \prim_n&= (-1)^{n}\omega_{n+1}.
\end{align}

Evaluating  eq. \eqref{eq:third} on $x\in \g$ we obtain $d\prim_1(x)=-\iota_{v_x}\omega$, which is equivalent to eq. \eqref{eq:mom}.

Evaluating eq. \eqref{eq:first} on $x_1,\dots,x_k\in \g$ we obtain
\begin{align*}
\sum_{1 \leq i < j \leq k}
(-1)^{i+j}\prim_{k-1}&([v_{x_i},v_{x_j}],v_{x_1},\ldots,\widehat{v_{x_i}},\ldots,\widehat{v_{x_j}},\ldots,v_{x_k})\\
&=-(-1)^k d\prim_{k}(v_{x_1},\ldots,v_{x_k})+ (-1)^{k-1}\iota(v_{x_1}\wedge \cdots \wedge v_{x_k})\omega.
\end{align*}
Multiplying this equation by $-\vs(k-1)=-(-1)^k\vs(k)$ we obtain eq. \eqref{main_eq_1}. Similarly one sees that eq. \eqref{eq:second}
is equivalent to eq. \eqref{main_eq_2}.
\end{proof} 

\begin{remark}
The results of this section can be derived also from \cite[\S 3]{FiorenzaRogersUrsPrequantum}. See \cite
{FRZ} for an explanation of how this derivation goes.
\end{remark}

\section{Closed forms and moment maps as cocycles}\label{sec:bar}

Recall that whenever
 $f: (A,d)\longrightarrow (A',d')$ is a map of complexes,  $(A[1]\oplus A', d_f)$ is a complex with differential
$d_f:=\left(\begin{smallmatrix}
 d & 0\\
 f & -d'
\end{smallmatrix}\right)$. This is known as the \emph{cone construction}.
 
We apply this to the map of complexes $\;\widetilde{}\;$ of Lemma \ref{tilde}.
We obtain:
\begin{prop}\label{prop:cone}
Fix an action of a Lie group $G$ on a manifold $M$. Then
$$\cB:=\Omega(M)^G[1]\oplus (\wedge^{\ge1} \g^*\otimes \Omega(M)),\;\;\;\;\; D:=
 \begin{pmatrix}
 d & 0\\
 \widetilde{} & -\dw
\end{pmatrix} $$ is a complex with the property: the $D$-closed elements in degree $n$ are pairs
$(\omega[1],\prim)$ where $\omega$ is a pre-$n$-plectic form
and $\prim$ corresponds (via Prop. \ref{lem:doublemomap}) to a moment map for $\omega$.
\end{prop}
\begin{proof}
Compute $D(\omega[1],\prim)=(d\omega[1],\widetilde{\omega}-\dw\prim)$ and apply
Prop. \ref{lem:doublemomap}.
\end{proof}

\section{Equivariant cohomology}\label{eqcohom}
 
In this section we recover in a quick way a result of \cite{FRZ}, which states that suitable extensions of $\omega$ in the Cartan model give rise to moment maps (see Prop. \ref{prop:concept}).  
  
 The following is a variation of Lemma \ref{tilde}: 
\begin{lemma}\label{tilde1} For any $G$-equivariant $F\colon \g \to \Omega^{N}(M)$, such that $\iota_{v_x}F(x)=0$ for all $x\in \g$,
 define $$ {F}_k\colon \wedge^k\g \to \Omega^{N+1-k}(M),\;\; {F}_k(x_1,\dots,x_k)=
 \iota(v_{x_1}\wedge\dots\wedge v_{x_{k-1}})F(x_k)$$ 
and
$\widetilde{F}:= {F}_1+\dots+ {F}_{N+1}\in \wedge^{\ge1} \g^*\otimes \Omega(M)$. 
Then $\dw\widetilde{F}=-\widetilde{d  F}$.  
\end{lemma}
\begin{remark}\label{rem:tilde1}
1) Notice that $F_1=F$.

2) If $\alpha\in \Omega^{N+1}(M)$ is $G$-invariant, then $\alpha_1
\colon \g \to \Omega(M)^{N}, x\mapsto \iota_{v_x}\alpha$ is $G$-equivariant since
$\cL_{v_{x}} (\iota_{v_y}\alpha)=\iota_{[v_x,v_y]}\alpha$.
We have $\widetilde{\alpha}=\widetilde{(\alpha_1)}$ (where the l.h.s. was defined in Lemma \ref{tilde} and the r.h.s. in Lemma \ref{tilde1}.)

3) $\widetilde{F}$ lies in the $G$-invariant part of $\wedge^{\ge1} \g^*\otimes \Omega(M)$, see \cite[\S 6]{FRZ} for a proof.
 \end{remark}
\begin{proof}
Notice first that the condition  $\iota_{v_x}F(x)=0$ 
ensures that 
$\widetilde{F}$ is well-defined (as a totally skew-symmetric map).
It also implies that $\iota_{v_x}d(F(x))=\cL_{v_x}(F(x))=F([x,x])=0$, so that $\widetilde{dF}$ is well-defined.

We compute for all $k$:
\begin{align}\label{eq:k11}
(d_{\g} {F}_{k-1})&(x_1,\dots,x_k)=\sum_{1\le i<j\le k}(-1)^{i+j}
 {F}_{k-1}([x_i,x_j],\dots,\widehat{x_i},\dots,\widehat{x_j},\dots,x_k)\\
=& 
\sum_{1\le i<j\le k-1}(-1)^{i+j}
\iota({[v_{x_i},v_{x_j}]\wedge\dots\wedge\widehat{v_{x_i}}\wedge\dots\wedge\widehat{v_{x_j}}\wedge\dots\wedge v_{x_{k-1}}})F(x_k) \nonumber\\
+&
\sum_{i=1}^{k-1}(-1)^{i+k}(-1)^{k-2} 
\iota({v_{x_1}\wedge\dots\wedge\widehat{v_{x_i}}\wedge \dots\wedge v_{{x_{k-1}}}})F([x_i,x_k])\nonumber  
\end{align}
and notice that $F([x_i,x_k])=\cL_{v_{x_i}}F(x_k)$ by the equivariance of $F$. Hence
\begin{align*}
(-1)^{k-1} d( {F}_k(x_1,\dots,x_k))&=(-1)^{k-1}d(\iota(v_{x_1}\wedge \dots\wedge v_{x_{k-1}})F(x_k))\\
&= (d_{\g}F_{k-1})(x_1,\dots,x_k)+\iota(v_{x_1}\wedge \dots\wedge v_{x_{k-1}})d(F(x_k)),
\end{align*}
where in the last equation we used Lemma \ref{tech_lemma} (applied to $\Omega:=F(x_k)$) and
eq. \eqref{eq:k11}.
  In other words:
 \begin{equation}\label{eq:inotherwords}
(-1)^{k-1}d F_k= d_{\g}F_{k-1}+ ({d  F})_k.
\end{equation} 
 
We conclude the proof computing
 \begin{align*}
\dw \widetilde{F}=
 \sum_{k=1}^{N+1}(d_\g{F}_k +(-1)^k d{F}_k)&=
 \sum_{k=2}^{N+2}d_\g{F}_{k-1} +\sum_{k=1}^{N+1}(-1)^k d {F}_k\\
&= \sum_{k=2}^{N+2}(d_\g{F}_{k-1} +(-1)^k d {F}_k)-d{F}_1 \\
&=
-\widetilde{d  {F}},
\end{align*}
using ${F}_{N+2}=0$ in the third equality and eq.
 \eqref{eq:inotherwords} in the last one.
 \end{proof}

Given the action of $G$ on $M$,
recall that the Cartan model is the complex\footnote{It calculates the equivariant cohomology when $G$ is compact.}
$(\Omega(M)\otimes S\g^*)^G$, where elements of $\g^*$ are assigned degree two, together with the Cartan differential $d_G$ (see for example \cite{GSSusy}). If we choose a basis $x_i$ of $\g$ and denote by $\xi^i$ the dual basis of $\g^*$ (concentrated in degree two), we can write $d_G=d\otimes 1-\sum_i\iota_{v_{x_i}}\otimes \xi^i$.

\begin{remark}\label{rem:CartanModel}
The invariant pre-$n$-plectic form $\omega$ (or, more precisely, $\omega\otimes 1$) is usually not closed in the Cartan model.
Given an equivariant linear map $\mu \colon \g\to \Omega^{n-1}(M)$, which we can regard as an element of $(\Omega^{n-1}(M)\otimes \g^*)^G$,
 we have \cite[\S 6.1]{FRZ}: $\omega-\mu$ is a closed element of the Cartan model if{f}  
 for all $x,y \in \g$
\begin{itemize}
\item[a)] 
$d\mu(x) = -\iota_{v_{x}} \omega$ \;\;(i.e., $v_x$ is the hamiltonian vector field of $\mu(x)$),
\item[b)] $\cL_{v_{x}} \mu(y ) = \mu([x,y])$\;\;  (i.e., $\mu \colon \g
\to \Omega^{n-1}_{ham}(M)$ is $G$-equivariant),
\item[c)]  $\iota_{v_{x}} \mu(x) =0$.
\end{itemize}
\end{remark}

We recover the main statement\footnote{In \cite[Thm. 6.3]{FRZ} it is
also shown that the moment map $f$ is equivariant,
using  b) in Remark \ref{rem:CartanModel}.}
 of \cite[Thm. 6.3]{FRZ}: 

\begin{prop}\label{prop:concept}
Let $\mu \colon \g\to \Omega^{n-1}(M)$ be an equivariant linear map so that $\omega-\mu$ is a cocycle in the Cartan model.
Then $\dw\widetilde{\mu}=\widetilde{\omega}$, and the maps ($1 \leq k \leq n$)
\begin{align*}
f_{k} \colon \wedge^k \g  &\to \Omega^{n-k}(M),\\
(x_{1},\ldots x_{k})&
\mapsto \vs(k)\iota(v_{x_{1}} \wedge\cdots \wedge
v_{x_{k-1}}) \mu(x_{k})
\end{align*}
are the components of a homotopy moment map $\g \to L_{\infty}(M,\omega)$.
\end{prop}
\begin{proof}
 By a) in Remark \ref{rem:CartanModel} we have the following equality of equivariant maps $\g\to \Omega^{n}(M)$:
$$d\mu=-\omega_1,$$
where $\omega_1(x)=\iota_{v_x}\omega$.
As $\iota_{v_x}\iota_{v_x}\omega=0$ for all $x$, we can apply the map $\;\widetilde{}\;$ (see Lemma \ref{tilde1}) to obtain
 $\widetilde{d\mu}=-\widetilde{\omega_1}$.
 We have $\dw\widetilde{\mu}=-\widetilde{d  \mu}$  
by  applying Lemma \ref{tilde1} to $\mu$
 (we are allowed to do so   because
of c) in Remark \ref{rem:CartanModel}), and we also have
 $\widetilde{\omega_1}=\widetilde{\omega}$ (see Rem. \ref{rem:tilde1}). Altogether 
 we obtain $$\dw\widetilde{\mu}=\widetilde{\omega}.$$   We conclude applying Prop. \ref{lem:doublemomap} to $\prim:=\widetilde{\mu}$.
\end{proof}

\begin{remark}
Prop. \ref{prop:concept} can be  extended \cite{AriasUribe}   \cite{FRZ}, as follows: every arbitrary extension of $\omega$ to a cocycle in the Cartan model gives rise to a moment map.
\end{remark} 
\section{Obstruction theory}\label{obstr}
 
We consider the obstruction theory for the existence of moment maps, obtaining results similar to those contained in 
\cite[\S 9.1, \S 9.2]{FRZ}.

Fix a point $p\in M$. 
It is immediate to check that 
$$r\colon (\wedge^{\ge1} \g^*\otimes \Omega(M),\dw)\to (\wedge \g^*,d_\g),\;\;\eta\otimes \alpha \mapsto \eta\cdot\alpha|_p$$
is a chain map. Here $\Omega|_p\in \RR$ is declared to vanish if $\Omega\in \Omega^{\ge 1}(M)$. 
Since $\widetilde{\omega}$ is $d_{tot}$-closed by Cor. \ref{omegacl},
it follows that $r(\widetilde{\omega})=(-1)^{n}{\omega}_{n+1}|_p\in \wedge^{n+1}\g^*$ is $d_\g$-closed, hence it defines a class in the Chevalley-Eilenberg cohomology $H_{CE}(\g)$.

\begin{cor}  Let $p\in M$. 
If a homotopy moment map exists, then $[{\omega}_{n+1}|_p]=0$.
\end{cor}
\begin{proof}
By Prop. \ref{lem:doublemomap},  a homotopy moment map exists if{f} $[\widetilde{\omega}]=0$. In this case $0=H(r)([\widetilde{\omega}])=(-1)^{n}[{\omega}_{n+1}|_p]$, where $H(r)$ denotes the map on cohomology induced by $r$.
\end{proof}

\begin{cor} If $[{\omega}_{n+1}|_p]=0$ and\footnote{These cohomology classes vanish, for example, if
$H^1(M)=\dots=H^n(M)=0$.} $$H^j_{CE}(\g)\otimes H^{n+1-j}(M)=0 \text{ for }j=1,\dots,n$$ 
then there exists a moment map.
\end{cor}
\begin{proof} The algebraic K\"unneth formula (\cite[Exerc. 14.23]{BT})
and the conditions on the vanishing of cohomology groups imply that 
$H(r)\colon H^{n+1}(\wedge^{\ge1} \g^*\otimes \Omega(M),\dw)\to H^{n+1}(\wedge \g^*,d_\g)$ is an isomorphism. Hence 
$[\widetilde{\omega}]$  vanishes if{f} $[{\omega}_{n+1}|_p]$   vanishes. 
The latter does vanish by assumption, so there exists $\prim$ with
$\dw\prim=\widetilde{\omega}$, and by Prop. \ref{lem:doublemomap} the primitive $\prim$ gives rise to a homotopy moment map.
\end{proof}

\section{Actions on mapping spaces}\label{LM}

In \cite[\S 11]{FRZ} it is shown that a moment map for a pre-$2$-plectic manifold $M$ gives rise to a moment map for the loop space $LM$ and an induced presymplectic form. Recall that  $LM=M^{S^1}$ consists of all differentiable maps from the circle $S^1$ to $M$. In this section we
  generalize this, allowing $M$ to be any pre-$n$-plectic manifold and 
 replacing the circle with any compact, orientable manifold.

\subsection{Loop spaces}\label{subsec:loop}

For the sake of exposition, consider first the case of the loop space $LM$ (an infinite-dimensional Fr\'echet manifold). The action of $G$ on $M$ induces an action on $LM$, simply given by $(g\cdot \gamma)(t):=g\cdot \gamma(t)$ for all $\gamma\in LM$ and $t\in S^1$. Given an element $x$ of the Lie algebra $\g$, recall that we denote by $v_x$ (a vector field on $M$) the corresponding infinitesimal generator  of the action on $M$. The  
corresponding infinitesimal generator of the action on $LM$, which we denote by $v_x^{\ell}$, is given as follows: $v_x^{\ell}|_{\gamma}=\gamma^*v_x\in \Gamma(\gamma^*TM)=T_{\gamma}LM$, for all $\gamma\in LM$.

There is a degree preserving\footnote{The notation 
$\Omega(LM)[-1]$ refers to the fact that here $\Omega^{k-1}(LM)$ is assigned degree $k$.}
 map 
\[
\ell \colon \Omega(M) \to \Omega(LM)[-1]
\]
called transgression, which commutes with the de Rham differential
\cite[\S 3.5]{brylinski}. Explicitly, it sends a form
 $\alpha\in \Omega^j(M)$ to  
$\alpha^{\ell}\in \Omega^{j-1}(LM)$ given by  
\[
\alpha^{\ell} \vert_{\gamma}(z_1,\ldots,z_{j-1}) = \int^{2\pi}_{0} 
\alpha(z_{1},\ldots,z_{j-1},\dot{\gamma}) \vert_{\gamma(s)} ~ ds
\quad \forall \gamma \in LM, ~ \forall z_{1},\ldots,z_{j-1} \in T_{\gamma}LM.
\] 
In particular, the closed  form $\omega\in \Omega^{n+1}(M)$   transgresses to a closed form $\omega^{\ell}\in \Omega^n(LM)$.

Consider the complex $\cC=(\wedge^{\ge 1} \g^*\otimes \Omega(M), \dw)$ of eq. \eqref{eq:dc}, as well as $\cC':=(\wedge^{\ge 1} \g^*\otimes \Omega(LM)[-1], \dw)$. The transgression map  extends trivially to 
a degree preserving map
  $$Id\otimes \ell\colon \cC\to \cC',$$
which commutes with $d_{\g}$  and the de Rham differential, and hence with $\dw$.
We use   the superscript $\ell$ to denote this map too. 
In particular, given
${\prim}=\prim_1+\dots+\prim_n\in \cC$ where $\prim_k \in \wedge^k \g^*\otimes \Omega^{n-k}(M)$, we obtain an element $\prim^{\ell}\in \cC'$ with components $\prim^{\ell}=(\prim^{\ell})_1+\dots+(\prim^{\ell})_{n-1}$ where $(\prim^{\ell})_k:=(\prim_k)^{\ell} \in \wedge^k \g^*\otimes \Omega^{n-k-1}(LM)$.

\begin{prop}\label{prop:transg}
If $\prim$ corresponds (in the sense of Prop. \ref{lem:doublemomap}) to a homotopy moment map for $(M,\omega)$, then $\prim^{\ell}$ corresponds to
 a homotopy moment map for $(LM,\omega^{\ell})$.
\end{prop}

\begin{proof}
If $\dw \prim=\widetilde{\omega}$ then
 \begin{equation}\label{ldtot}
\dw (\prim^{\ell})=(\dw \prim)^{\ell}=(\widetilde{\omega})^{\ell}=\widetilde{{\omega}^{\ell}}.
\end{equation}
The last equality holds because for all $k$ we have  $({\omega}_k)^{\ell}=({\omega}^{\ell})_k$, as a consequence of
\begin{align*}
({\omega}_k)^{\ell} (x_1,\dots,x_k) \vert_{\gamma}=(\iota{(v_{x_1}\wedge\dots\wedge v_{x_k})}\omega)^{\ell} \vert_{\gamma}=&
\int^{2\pi}_{0} 
\omega(v_{x_1},\ldots,v_{x_k},\bullet,\dot{\gamma}) \vert_{\gamma(s)} ~ ds\\
=& \iota{(v_1^{\ell}\wedge\dots\wedge v_k^{\ell})}({\omega}^{\ell})\vert_{\gamma}
=
({\omega}^{\ell})_k (x_1,\dots,x_k) \vert_{\gamma},
\end{align*}
where $x_1,\dots,x_k \in \g$, $\gamma\in LM$, and $\bullet$ denotes  $n-k$ slots for elements of  $T_{\gamma}LM$. Recall that $({\omega}^{\ell})_k$ was defined in Lemma \ref{tilde}.
   
Thanks to eq. \eqref{ldtot} we can now apply
Prop. \ref{lem:doublemomap} (which holds in the setting\ of Fr\'echet manifolds too). 
\end{proof}

\subsection{General mapping spaces}

We now generalize Prop. \ref{prop:transg}. Let $\Sigma$ be a compact, oriented manifold of dimension $s$.
The $G$ action on $M$ gives rise to a $G$ action on
$M^{\Sigma}$, 
 the Fr\'echet manifold of smooth maps from $\Sigma$ to $M$.
(The formulae for the this action and the corresponding infinitesimal generators are exactly as in \S \ref{subsec:loop}).

Transgression is the differential-preserving map
\[
\ell:=\int_{\Sigma}\circ\; ev^* \colon \Omega(M) \to \Omega(M^{\Sigma})[-s]
\] 
where 
$ev\colon \Sigma \times M^{\Sigma}\to M$
is the evaluation map and 
$\int_{\Sigma}$ denotes  integration along the fibers \cite[Cap. VI.4]{audin2004torus} of the projection $\Sigma \times M^{\Sigma}\to   M^{\Sigma}$, which lowers the degree of a differential form by $s$.
Notice that  the closed  form $\omega\in \Omega^{n+1}(M)$   transgresses to a closed  form $\omega^{\ell}\in \Omega^{n+1-s}(M^{\Sigma})$.

The transgression map  extends trivially to a map of complexes  $$Id\otimes \ell\colon \cC\to \cC',$$ where now
 $\cC':=(\wedge^{\ge 1} \g^*\otimes \Omega(M^{\Sigma})[-s], \dw)$.
Let ${\prim}=\prim_1+\dots+\prim_n$, where $\prim_k \in \wedge^k \g^*\otimes \Omega^{n-k}(M)$.
\begin{prop}\label{prop:corr}
If $\prim$ corresponds (in the sense of Prop. \ref{lem:doublemomap}) to a homotopy moment map for $(M,\omega)$, then $\prim^{\ell}$ corresponds to
 a homotopy moment map for $(M^{\Sigma},\omega^{\ell})$.
\end{prop}
\begin{proof}
The proof is the same as for Prop. \ref{prop:transg}. We just point out that the equation $({\omega}_k)^{\ell}=({\omega}^{\ell})_k$ for all $k$ is obtained applying the well-known relation $(\iota_{v_x}\omega)^{\ell}=\iota_{(v_x)^{\ell}}{\omega}^{\ell}$ for all $x\in \g$.  
It
can be proven also directly, exactly as in the proof of Prop.\ref{prop:transg},  using the explicit description of the integration along the fiber given in \cite[Cap. VI.4]{audin2004torus} and the fact that the derivative $d_{(p,\sigma)}ev\colon T_{(p,\sigma)}(\Sigma \times M^{\Sigma})\to T_{\sigma(p)}M$
maps
a tangent vector of the form $(Z,0)$ to the vector $\sigma_*(Z)$ and, for all $x\in \g$, the vector  $(0,(v_x)^{\ell})|_{(p,\sigma)}$ to $(v_x)|_{\sigma(p)}$.
\end{proof}

Spelling out Prop. \ref{prop:corr} in terms of moment maps  we obtain:
\begin{cor}
Let  $$f\colon \g \to L_{\infty}(M,\omega)$$ 
be a homotopy moment map for the pre-$n$-plectic manifold $(M,\omega)$ with components 
$f_k\colon \wedge^k\g\to \Omega^{n-k}(M)$, where $k=1,\dots,n$.

Then $$f^{\ell}\colon \g \to L_{\infty}(M^{\Sigma},\omega^{\ell})$$ 
is a homotopy moment map for the pre-$(n-s)$-plectic manifold
$(M^{\Sigma},\omega^{\ell})$ with components 
$(f^{\ell})_k:=(f_k)^{\ell}\colon \wedge^k\g\to \Omega^{n-s-k}(M^{\Sigma})$, where $k=1,\dots,n-s$.
\end{cor}

\section{Equivalences}\label{equiv}
  
In this section we introduce notions of equivalence  for: 1)  certain
cocycles in the Cartan model, and 2) pairs consisting of closed invariant forms  and moment maps. 
 Recall that
Prop. \ref{prop:concept} states that to the former Cartan cocycles one can canonically associate moment maps; we show that the above equivalences are compatible with this association. 
All along this section we fix
an action of a  Lie group $G$  on a manifold $M$.

 \subsection{Equivalences of Cartan cocycles}

\begin{defi}\label{def:ecequiv} Two  cocycles $C^0=\omega^0-\mu^0$ and $C^1=\omega^1-\mu^1$ in the Cartan model, with $\omega^0,\omega^1\in\Omega^{n+1}(M)^G$ and 
$\mu^0,\mu^1\in (\Omega^{n-1}(M)\otimes \g^*)^G$, are \textbf{equivalent} if{f} they differ by a coboundary of the form\footnote{If $C^0-C^1$ is exact, in general we may not find a primitive of the form $\alpha+F$ as above. We justify our definition remarking that the choices of Cartan cocycles we allow are also not the most general ones.
}
 $d_G
(\alpha+F)$, where
$$\text{   
$\alpha \in \Omega^{n}(M)^G$ and $F\in (\Omega^{n-2}(M)\otimes \g^*)^G$}.$$

Explicitly, $C^1-C^0=d_G(\alpha+F)$ means that
\begin{itemize}
\item[a)] $\omega^1-\omega^0=d\alpha$
\item[b)] $\mu^1-\mu^0=\iota_{\bullet}\alpha-dF$
\item[c)] $\iota(v_x)F(x)=0$
for all $x\in \g$,
\end{itemize}
where we use the short form $(\iota_{\bullet}\alpha)(x):=\iota_{v_x}\alpha$.  
\end{defi}

\begin{remark}
In the symplectic case (so $n=1$), Def. \ref{def:ecequiv} reduces to: 
$\omega^1-\omega^0=d\alpha$ and $\mu^1-\mu^0=\iota_{\bullet}\alpha$ for some
$\alpha \in \Omega^{1}(M)^G$. In particular, if $\omega^1=\omega^0$, each function $\iota_{v_x}\alpha$ is constant. 
  \end{remark}
  
The following proposition states that if two Cartan cocycles are related by a $G$-equivariant diffeomorphism of $M$ isotopic to the identity, then they are equivalent.

\begin{prop}\label{prop:geoequivcoho}
Let the Lie group $G$ act on $M$. Let $\omega^0,\omega^1\in\Omega^{n+1}(M)^G$. Let $\mu^0,\mu^1\in (\Omega^{n-1}(M)\otimes \g^*)^G$
so that $C^i:=\omega^i-\mu^i$ is a cocycle in the Cartan model.
 Suppose that there exists a $G$-equivariant diffeomorphism $\psi$, isotopic to $Id_M$ by $G$-equivariant diffeomorphisms, with  
\begin{equation*}
\psi^*\omega^1=\omega^0,\;\;\;\psi^*\mu^1=\mu^0.
\end{equation*}
(Here $\mu^1$ is viewed as a map $\g\to \Omega^{n-1}(M)$ and
$(\psi^*\mu^1)(x):=(\psi^*(\mu^1(x))$ for all $x\in \g$).
Then $C^1$ and $C^0$ are equivalent in the sense of Def. \ref{def:ecequiv}.
\end{prop}
\begin{proof}
We construct explicitly equivariant Cartan cochains $\alpha, F$ such that $d_G(\alpha+F)=C^1-C^0$.

Let $\{\psi_s\}_{s\in[0,1]}$ a isotopy from $\psi^0=Id_M$  to $\psi=\psi^1$  by  $G$-equivariant diffeomorphisms, and denote by 
$\{X_s\}_{s\in[0,1]}$ the time-dependent vector field generating $\{\psi_s\}_{s\in[0,1]}$. Define 
$\omega^s$ by $\psi_s^*(\omega^s)=\omega^0$ and 
$\mu^s\in (\Omega^{n-1}(M)\otimes \g^*)^G$ by $\psi_s^*(\mu^s)=\mu^0$. 

We claim that 
$$\alpha:=-\int_0^1\iota_{X_s}\omega^s\in \Omega^n(M) $$
satisfies $\omega^1-\omega^0=d\alpha$ (condition a) in Def. \ref{def:ecequiv}). This follows integrating from $s=0$ to $s=1$ the equation 
$$\frac{d}{ds}\omega^s=-\cL_{X_s}\omega^s=-d\iota_{X_s}\omega^s$$
where in the first equality we use \cite[Prop. 6.4]{Ana} 
\begin{equation*}
\label{eq:ana}
0=\frac{d}{ds}(\psi_s^*\omega^s)=\psi_s^*(\cL_{X_s}\omega^s+\frac{d}{ds}\omega^s).
\end{equation*}

We claim that 
$$F:=\int_0^1\iota_{X_s}\mu^s\in \Omega^{n-2}(M)\otimes \g^*$$
satisfies $\mu^1-\mu^0=\iota_{\bullet}\alpha-dF$ (condition b) in Def. \ref{def:ecequiv}).
Similarly to the above, this follows integrating from $0$ to $1$ the following expression, for all $x\in \g$:  
\begin{align*}
\frac{d}{ds}\mu^s(x)=-\cL_{X_s}(\mu^s(x))&=-\iota_{X_s}(d\mu^s(x))
-d\iota_{X_s}\mu^s(x)\\
&=\iota_{X_s}(\iota_{v_x}\omega^s)-d\iota_{X_s}\mu^s(x)\\
&=\iota_{v_x}(-\iota_{X_s}\omega^s)-d\iota_{X_s}\mu^s(x).
\end{align*}
Here in the first equality we used again \cite[Prop. 6.4]{Ana}, and in the third one   $\iota_{v_x}\omega^s=-d(\mu^s(x))$ for all $x\in \g$ (see  Remark \ref{rem:CartanModel} a)).

We are left with showing $\iota_{v_x}F(x)=0$ for all $x\in \g$ (condition c) in Def. \ref{def:ecequiv}). This holds since
$\iota_{v_x}\mu^s(x)=0$, a consequence  of   Remark \ref{rem:CartanModel} c) and the fact that $\omega^s-\mu^s$, being the pullback of a Cartan cocycle by a $G$-equivariant map,  is itself a Cartan cocycle. 

Notice that  the $X_s$ are $G$-invariant, as their flow $\{\psi_s\}$ commutes with the $G$-action. Further 
the $\omega^s$ are $G$-invariant, since $\omega$ is $G$-invariant.
Hence
$\alpha$ is $G$-invariant. By the same reasoning and the invariance of $\mu^0$, we see that  $F$ is $G$-invariant.
\end{proof}

\subsection{Equivalences of moment maps} 
 
Let $f \colon \g \rightsquigarrow L_{\infty}(M,\omega)$ be a homotopy moment map for $\omega$, and   
\begin{equation}\label{eq:muf}
\prim_k :=\varsigma(k) f_k \colon \wedge^k\g \to \Omega^{n-k}(M)\;\;\;\;\;\; \text{for}\;k=1,\dots,n.
\end{equation}
We know that
$\prim=\prim_1+\dots+\prim_n$ satisfies $\dw\prim=\widetilde{\omega}$.

Indeed, by   Prop. \ref{lem:doublemomap}, this equation characterizes moment maps for $\omega$. Therefore, if $\eta\in (\wedge^{\ge 1} \g^*\otimes \Omega(M))_{n-1}$, then $\prim+\dw \eta$ naturally provides a new moment map for $(M,\omega)$. Further, notice that if $\alpha\in (\Omega^n)^G$, by Lemma \ref{tilde} we have $\dw(\prim+\widetilde{\alpha})=\widetilde{\omega+d\alpha}$, i.e. $\prim+\widetilde{\alpha}$ provides a moment map for 
$\omega+d\alpha$. The following definition,
which arises naturally considering the complex $\cB$ introduced in \S \ref{sec:bar},
 is made so that these two kinds of moment maps are equivalent to the original one.

\begin{defi}\label{def:eqdc}
Let $\omega$ be an invariant pre-$n$-plectic form on $M$ and $f$ a moment map for $\omega$, for which we denote by $\prim$ the corresponding element of $(\wedge^{\ge 1} \g^*\otimes \Omega(M))_{n}$
as in Prop. \ref{lem:doublemomap},
 and similarly for $(\omega',f')$.
The pairs $(\omega,f)$ and $(\omega',f')$  
are \textbf{equivalent} 
  if there exist $\eta\in (\wedge^{\ge 1} \g^*\otimes \Omega(M))_{n-1}$ and $\alpha\in (\Omega^n)^G$ such that
\begin{align}\label{eq:dgalpha}
\omega'-\omega&=d\alpha\\
\label{eq:dgeta}
\prim'-\prim &=\dw \eta+\widetilde{\alpha}.
\end{align}
\end{defi}
 \begin{remark}
The equivalence introduced in Def. \ref{def:eqdc} can be phrased as a simple coboundary condition, thereby providing an algebraic justification for Def. \ref{def:eqdc}. Indeed,
in terms of the complex $\cB=(\Omega(M)^G[1]\oplus (\wedge^{\ge1} \g^*\otimes \Omega(M)), D)$ introduced in \S \ref{sec:bar}, the conditions \eqref{eq:dgalpha}-\eqref{eq:dgeta} are simply 
phrased as $$(\omega'[1],\prim')-(\omega[1],\prim)=D(\alpha[1],-\eta).$$

A geometric justification for Def. \ref{def:eqdc} is given in Prop. \ref{prop:geomjust}.

In  Appendix \ref{sec:dolg} we compare Def. \ref{def:eqdc} with the natural notion of equivalence for $L_{\infty}$-morphisms. There we show that 
two homotopy moment maps  are equivalent   \emph{with $\alpha=0$} (see Def. \ref{def:eqdc}) if{f} they are equivalent as $L_{\infty}$-morphisms. 
\end{remark}
 
\begin{remark} \label{lem:expleq}
Condition \eqref{eq:dgeta}, explicitly, is that   there exists $\eta=\eta_1+\dots+\eta_{n-1}\in (\wedge^{\ge 1} \g^*\otimes \Omega(M))_{n-1}$ and $\alpha\in (\Omega^n)^G$
  with 
\begin{align*}
-d\eta_1+\alpha_1&=(\prim'-\prim)_1,\\
d_{\g}\eta_{k-1}+(-1)^{k} d\eta_{k}+(-1)^{k-1}\alpha_k&= (\prim'-\prim)_k\;\;\;\; \forall k=2,\dots,n-1\\
d_{\g}\eta_{n-1}+(-1)^{n-1}\alpha_{n}&=(\prim'-\prim)_n,
\end{align*}
where $\alpha_i$ is defined as in  eq. \eqref{eq:omegak}.
 \end{remark}

\begin{remark}
Given a $G$-manifold, we can consider $\omega=0\in \Omega^{n+1}_{closed}(M) $ and the zero moment map $f=0$. Given $ \alpha \in \Omega^n (M)^G$, applying the operation described in  Def. \ref{def:eqdc} provides a moment map for $(M,d\alpha)$, which agrees exactly with the one provided in \cite[\S 8.1]{FRZ} for exact pre-$n$-plectic forms admitting an invariant primitive.
\end{remark}

The following proposition provides a geometric justification for Def. \ref{def:eqdc}. It states that if two moment maps are related by a $G$-equivariant diffeomorphism of $M$ isotopic to the identity, then they are equivalent.
  \begin{prop}\label{prop:geomjust}
Let the Lie group $G$ act on $M$. Let $\omega^0,\omega^1$ be   closed $(n+1)$-forms preserved by the action. Suppose that there exists a $G$-equivariant diffeomorphism $\psi$, isotopic to $Id_M$ by $G$-equivariant diffeomorphisms, such that $\psi^*\omega^1=\omega^0$.
 
 Let $f^i\colon \g\rightsquigarrow L_{\infty}(M,\omega_i)$ be homotopy moment maps $(i=0,1)$ intertwined by $\psi$ that is, for all their components ($k=1,\dots,n$) we have
\begin{equation*}
\psi^*f_k^1=f_k^0.
\end{equation*} 
Then $f^0$ and $f^1$ are   equivalent in the sense of Def. \ref{def:eqdc}.
\end{prop}

\begin{proof} 
Let $\{\psi_s\}_{s\in[0,1]}$ a isotopy from  $\psi^0=Id_M$ to $\psi=\psi^1$ by  $G$-equivariant diffeomorphisms, and denote by 
$\{X_s\}_{s\in[0,1]}$ the time-dependent vector field generating $\{\psi_s\}_{s\in[0,1]}$. Define 
$\omega^s$ by $\psi_s^*(\omega^s)=\omega^0$ and 
$f^s $ by $\psi_s^*(f^s)=f^0$. 

The form
$$\alpha:=-\int_0^1\iota_{X_s}\omega^s \;ds\in \Omega^n(M) $$
satisfies $\omega^1-\omega^0=d\alpha$, i.e. eq. \eqref{eq:dgalpha},
as we have already shown at the beginning of the proof of Prop. \ref{prop:geoequivcoho}.
 
Now, for all $1\le k \le n$ (and defining $f^0=0$) we have  
\begin{align*}
\frac{d}{ds}f^s_k=-\cL_{X_s}f^s_k&=-d(\iota_{X_s}f^s_k)-\iota_{X_s}df^s_k\\
&= 
-d(\iota_{X_s}f^s_k)+\iota_{X_s}d_{\g}f^s_{k-1}+
\vs(k)\iota_{X_s}\omega^s_k\\
&= 
-d(\iota_{X_s}f^s_k)+d_{\g}\iota_{X_s}f^s_{k-1}+
(-1)^k\vs(k)(\iota_{X_s}\omega^s)_k,
\end{align*}
where in the first equation we used   \cite[Prop. 6.4]{Ana}, and in the second one
  $df^s_k=-d_{\g}f_{k-1}^s-\vs(k)\omega^s_k$
(which holds by eq. \eqref{eq:first}).

Multiplying by $\vs(k)$ the equation $f_k^1-f_k^0=\int_0^1 \frac{d}{ds}f^s_k\;ds$ we hence obtain
$$\prim_k^1-\prim_k^0=(-1)^k d\eta_k + d_{\g} \eta_{k-1} +(-1)^{k-1} \alpha_k,$$
where we define
$$\eta_k\colon \wedge^k\g \to \Omega^{n-1-k}(M),\;\;\eta_k(x^1,\dots,x_k)=(-1)^{k-1}
\vs(k)\int_0^1\iota(X_s)f_k^s \;ds.$$ 
As this holds for all $1\le k\le n$, we obtain $\prim_1-\prim_0=\dw \eta+\widetilde{\alpha}$, i.e. eq. \eqref{eq:dgeta}.
\end{proof}

We finish this subsection discussing equivalences for which the pre-$n$-plectic form  is fixed.
 Fix a pre-$n$-plectic form $\omega$. Restricting Def. \ref{def:eqdc} to the space of moment maps for $\omega$ we obtain: two moment maps $f,f'$ for $\omega$ are equivalent if{f}  
 there
 exist $\eta\in (\wedge^{\ge 1} \g^*\otimes \Omega(M))_{n-1}$ and a \emph{closed} form $\alpha\in (\Omega^n)^G$ such that
eq. \eqref{eq:dgeta} is satisfied\footnote{Loosely speaking, the action of $\alpha$ can be interpreted as induced by a gauge transformation on the higher Courant algebroid $TM\oplus \wedge^{n-1}T^*M$ endowed with the $\omega$-twisted Courant bracket.}.

The following proposition extends the results of \cite[\S 7.5]{FRZ}. 
\begin{prop}\label{prop:fixomega}  There exist a closed $3$-form  $\omega$ with the following property: 
There exist  an \emph{equivariant} moment map for $\omega$ which is not equivalent (in the sense of Def. \ref{def:eqdc}) to any moment map for $\omega$ arising from a Cartan cocycle of the form $\omega-\mu$ as in Prop. \ref{prop:concept}. 
\end{prop}
\begin{proof}  
Consider an action of a Lie group $G$ on a connected   pre-$2$-plectic manifold $(M,\omega)$, and let $f$ be an {equivariant} moment map.
Let $f'$ be another  {equivariant} moment map (for the same action on $(M,\omega)$) which is equivalent to $f$. This means exactly that there exist  $\eta_1\in \g^*\otimes C^{\infty}(M)$ and a \emph{closed} form $\alpha\in \Omega^2(M)^G$ satisfying  eq. \eqref{eq:dgeta}. In particular the   equation 
\begin{equation}\label{eq:impce}
-d\eta_1+\alpha_1=(\prim'-\prim)_1
\end{equation}
holds, where we denote   $\prim_1=f_1$,$\prim_2=f_2$, etc.

For every $x\in \g$, evaluating the l.h.s. of the eq. \eqref{eq:impce} on $x$ and applying  the interior product  $\iota_{v_x}$, we obtain a function. We claim that this function is a constant, which we denote by $C_x$. To see this, first notice that 
evaluating the l.h.s. of the eq. \eqref{eq:impce} on $x$ and applying    $\iota_{v_x}$ we obtain
$\iota_{v_x}(-d\eta_1(x)+\alpha_1(x))=-\cL_{v_x}(\eta_1(x))$.
Second, notice that  since $\prim_1,(\prim')_1,\alpha_1$ are equivariant, it follows that  
$d\eta_1$ is also equivariant. In particular  we have $d\cL_{v_x}\eta_1(x)=\cL_{v_x}d\eta_1(x)=d\eta_1([x,x])=0$, i.e., $\cL_{v_x}(\eta_1(x)):=-C_x$ is a  constant function.

From the above claim we conclude: if there exists $x\in \g$ such that
\begin{enumerate}
\item[i)] $\iota_{v_x}\prim_1(x)\neq 0$
\item[ii)] $C_x=0$
\end{enumerate}
  then necessarily $\iota_{v_x}(\prim')_1(x)\neq 0$, so
$f'$ can not arise from a Cartan cocycle (compare with Remark \ref{rem:CartanModel} c)). 

Now, following \cite[\S 7.5]{FRZ}, we display an example of moment map $f$ and   $x\in \g$    satisfying assumption i) and such that for \emph{every} 
 $\eta_1\in \g^*\otimes C^{\infty}(M)$ assumption ii) is satisfied. It follows that there exists no moment map which is  equivalent to $f$ and which  arises from a Cartan cocycle.

%
%
%
%
 
Let $G$ be the abelian group $S^1\times S^1$, and $(M,\omega)=(S^1\times S^1 \times \RR, d\theta_1\wedge d\theta_2 \wedge dz)$. We take the infinitesimal action of $\g$ on $M$ to be 
given by $(1,0)\in \g \mapsto \partial_{\theta_1}$, $(0,1)\mapsto \partial_{\theta_2}$.
It is easily checked that 
\begin{align*}
f_1\colon \g \to \Omega^1_{Ham}(M),&\;\;\;\;\;
(1,0)\mapsto zd\theta_2+d\theta_1,\;\;\; (0,1)\mapsto -zd\theta_1+d\theta_2,\\
f_2\colon \wedge^2 \g \to C^{\infty}(M),&\;\;\;\;\; (1,0)\wedge (0,1)\mapsto -z
\end{align*}
 are the components of an equivariant moment map.
Let  $x:=(1,0)\in \g$. Since $v_x=\partial_{\theta_1}$,
we clearly have $\iota_{v_x}f_1(x)=1\neq 0$, hence assumption i) is satisfied.
For \emph{any}
$h\in C^{\infty}(M)$  such that $\cL_{v_x}(h)=\partial_{\theta_1}(h)$ is a constant, integrating $\cL_{v_x}(h)d\theta_1$ along the circles $S^1\times \{point\}\times \{point\}$ of $M$ one sees by Stokes' theorem that the constant $\cL_{v_x}(h)$ is  necessarily  zero. Hence,
 for  \emph{any} $\eta_1\in \g^*\otimes C^{\infty}(M)$ 
 we have $\cL_{v_x}(\eta_1(x))=0$, verifying that  
  assumption ii) is satisfied. 
\end{proof}

\begin{remark}\label{rem:inner}
The notion of equivalence on the space of moment maps for $\omega$ 
mentioned just before Prop. \ref{prop:fixomega}
should not be confused with the similar but more restrictive one in which  $\alpha=0$ is imposed.
We refer to the latter notion of equivalence  as {\bf inner equivalence}. Explicitly: 
two moment maps
$f$ and $f'$   for $\omega$ are inner equivalent if there exist $\eta\in (\wedge^{\ge 1} \g^*\otimes \Omega(M))_{n-1}$ such that $\prim'-\prim =\dw \eta$, where  $\prim$ denotes the  element of $(\wedge^{\ge 1} \g^*\otimes \Omega(M))_{n}$ corresponding to $f$ 
as in Prop. \ref{lem:doublemomap},
 and similarly for $\prim'$ and $f'$. The notion of inner equivalence arises naturally if one considers the complex $(\cC,\dw)$ of \S \ref{momap} (as opposed to the  complex $\cB$ introduced in \S \ref{sec:bar}).

Notice that when $\alpha=0$,
 the first equation in Rem. \ref{lem:expleq} says that, for all $x\in \g$, the elements
$(\prim')_1(x)$ and $\prim_1(x)$ of $\Omega^{n-1}_{ham}(M,\omega)$ -- which have the same hamiltonian vector field $v_x$, and hence a priory differ by a closed form -- actually differ by an exact form.
\end{remark}

\subsection{Relation between the two notions of equivalence}
We end establishing the relation between the equivalences introduced in  
Def. \ref{def:ecequiv} and Def. \ref{def:eqdc}.

\begin{prop} Let the Lie group $G$ act on $M$. 
Take two  Cartan cocycles $C^0=\omega^0-\mu^0$ and $C^1=\omega^1-\mu^1$, with $\omega^0,\omega^1\in \Omega^{n+1}(M)^G$ and 
$\mu^0,\mu^1\in (\Omega^{n-1}(M)\otimes \g^*)^G$. Assume that $C^0$ and $C^1$ are equivalent in the sense of Def. \ref{def:ecequiv}.

Then the homotopy moment maps $f^0$ and $f^1$, 
induced by the $\mu^i$ as in Prop. \ref{prop:concept},
are equivalent in the sense of Def. \ref{def:eqdc}.
\end{prop}
\begin{proof}
Since we assume that $C^0$ and $C^1$ are equivalent,   there is  $\alpha \in \Omega^{n}(M)^G$ and $F\in (\Omega^{n-2}(M)\otimes \g^*)^G$ 
satisfying the  equation  appearing below Def. \ref{def:ecequiv}, that is
\begin{itemize}
\item[a)] $\omega^1-\omega^0=d\alpha$
\item[b)] $\mu^1-\mu^0=\iota_{\bullet}\alpha-dF $
\item[c)] $\iota(v_x)F(x)=0$
for all $x\in \g$,
\end{itemize} 
We now check that an equivalence between the homotopy moment maps is given by the form $\alpha$ and by $\eta:=\widetilde{F}$ (notice that $\widetilde{F}$ is well-defined by c)). The relation \eqref{eq:dgalpha} is just a).

Applying the map $\;\widetilde{}\;$ (see Lemma \ref{tilde1}) to  equation b) we obtain
$$\widetilde{\mu^1}-\widetilde{\mu^0}=\widetilde{\iota_{\bullet}\alpha}-\widetilde{dF}.$$
 Denoting by $\prim^i\in  (\wedge^{\ge 1} \g^*\otimes \Omega(M))_{n}$ the element corresponding   to  $f^i$ as in Prop. \ref{lem:doublemomap}, for $i=0,1$, 
we have $\prim^i= \widetilde{\mu^i}$ (to see this, compare the formulae in Prop. \ref{lem:doublemomap} and Prop. \ref{prop:concept}). Using $\widetilde{\iota_{\bullet}\alpha}=\widetilde{\alpha}$ (by Rem. \ref{rem:tilde1})
and $\dw\widetilde{F}=-\widetilde{d  F}$ (by Lemma \ref{tilde1}) we obtain exactly   the relation \eqref{eq:dgeta}. \end{proof}

\appendix

\section{Equivalences of moment maps and $\li$-algebra morphisms}\label{sec:dolg}
 
 
Let $G$ be a Lie group acting on a pre-$n$-plectic manifold $(M,\omega)$.
A moment map for this action (Def. \ref{def:momap}) is in particular an 
$\li$-morphism $\g\to \li(M,\omega)$. 
There is a natural notion of equivalence of $\li$-morphisms, and the aim of this appendix is to show that it 
coincides with the inner equivalence introduced in Rem. \ref{rem:inner} (that is,  equivalence in the sense of Def. \ref{def:eqdc} imposing $\alpha=0$.) 

The notion of equivalence of $\li$-morphisms
 comes from homotopy theory, and 
 coincides with the one given by equivalences of Maurer-Cartan elements 
\cite{DolgErr}. We express it following \cite[\S 5]{DolgErr}: let $\widetilde{L},L$ be $\li$-algebras. Consider $\Omega(\RR)=\RR[t]+\RR[t]dt$, the differential graded algebra of polynomial forms on the real line, where $t$ has degree 0 and $dt$ degree 1. Then $L\otimes\Omega(\RR)$ is again an $\li$-algebra \cite[\S1]{Buijs}.
\begin{defi}\label{def:eqdolg} Let $\widetilde{L},L$ be $\li$-algebras.
Let $f,f'\colon \widetilde{L}\to L$ be $\li$-morphisms.
 $f$ and $f'$ 
 are \textbf{equivalent} if{f}  there exists an $\li$-morphism $H \colon \widetilde{L}\to L\otimes\Omega(\RR)$ such that 
\begin{equation}\label{H10}
H|_{t=0,dt=0}=f,\;\; H|_{t=1,dt=0}=f'.
\end{equation}
\end{defi}

\begin{prop}
Two homotopy moment maps $f$ and $f'$ are inner equivalent (see Rem. \ref{rem:inner}) if{f} they are equivalent in the sense of Def. \ref{def:eqdolg}.
\end{prop}
\begin{proof}
We first given a characterization of $L_{\infty}$-morphisms from $\g$ to $\li(M,\omega)\otimes\Omega(\RR)$.
$\li(M,\omega)\otimes\Omega(\RR)$ is concentrated in degrees $\le 1$, and its  multibrackets   $\{l_k\}$ are as follows \cite[\S1]{Buijs}: for $k\ge 2$ they
are given by the multibrackets of $\li(M,\omega)$ extended by $\Omega(\RR)$-linearity (with no signs involved),
and the differential   is $$l_1(\gamma\otimes \Gamma)=D\gamma\otimes \Gamma+(-1)^{deg(\gamma)} 
\gamma\otimes \frac{\partial}{\partial t}\Gamma dt,$$
where $D$ denotes the differential in $\li(M,\omega)$.
All multibrackets, except for the differential, vanish unless all entries are in degree $0$ or $1$.

 We observe that  the truncation
 $$T:=(\li(M,\omega)\otimes\Omega(\RR))_{< 0}\oplus \{c\in (\li(M,\omega)\otimes\Omega(\RR))_{0}: l_1(c)=0\}$$  is closed\footnote{Indeed, $T$ is closed under $l_1$ since $(l_1)^2=0$. For the higher brackets, the only non-trivial case to consider is $l_2(\gamma\otimes \Gamma,\gamma'\otimes \Gamma')$ when $\gamma,\Gamma,\gamma',\Gamma'$ all have degree zero. This bracket lies in $T$ since $l_1$ satisfies the Leibniz rule w.r.t. $l_2$.}
under the multibrackets.
 Hence $T$ is a $L_{\infty}$-algebra for which the multibrackets (except for the differential) vanish unless all entries are in degree zero.

  Let $$H\colon \wedge^{\ge 1}\g \to \li(M,\omega)\otimes\Omega(\RR)$$ be a linear   map such that $H|_{\wedge^{  k}\g}$ has degree $1-k$. \\ 
  
\noindent\textit{\underline{Claim}:  $H$ is an $L_{\infty}$-morphism if{f} 
conditions
\eqref{uno} and \eqref{due} below are satisfied.}
  
We divide the proof of the claim in three steps.

A)  
$H$ is a $\li$-morphism if{f} the image of the first component  $H_1$ is annihilated by $l_1$ and for $2 \leq m \leq n+1$,
  for all $x_{i} \in \g$
\begin{multline}\label{eq:P}  
\sum_{1 \leq i < j \leq m}
(-1)^{i+j+1}H_{m-1}([x_{i},x_{j}],x_{1},\ldots,\widehat{x_{i}},\ldots,\widehat{x_{j}},\ldots,x_{m})\\
=l_{1} H_{m}(x_{1},\ldots,x_{m}) + l_{m}(H_{1}(x_{1}),\ldots,H_{1}(x_{m})) \end{multline}
 where $H_{n+1}=0$. 
Indeed, if $H$ is a $\li$-morphism, then $H_1$ is a chain map and takes values in $l_1$-closed elements,  and therefore $H$ takes values in $T$, so that we can apply \cite[\S 3.2]{FRZ}. 
Conversely, if the image of $H_1$ is annihilated by $l_1$, we can apply 
\cite[\S 3.2]{FRZ}, and eq. \eqref{eq:P}  implies that $H$ is a $\li$-morphism.


B) Write
$$H= h^0(t)+h^1(t)dt$$ 
where $h^0(t)$ and $h^1(t)$ are maps $\wedge^{\ge 1}\g\to \li(M,\omega)\otimes \RR[t]$. Notice that the component $h^0(t)_k$ has degree $1-k$, while $h^1(t)_k$ has degree $-k$. The condition that the degree zero component of $H_1$ takes values in $l_1$-closed elements reads
\begin{equation}\label{eq:nol1}
\frac{\partial}{\partial t}h^0(t)_{1}+dh^1(t)_{1}=0.
\end{equation}
Separating the terms without $dt$ from those containing $dt$,
eq. \eqref{eq:P}   is equivalent to ($2 \leq m \leq n+1$)
\begin{multline} 
 \label{eq:multi1}
\sum_{1 \leq i < j \leq m}
(-1)^{i+j+1}h^0(t)_{m-1}([x_{i},x_{j}],x_{1},\ldots,\widehat{x_{i}},\ldots,\widehat{x_{j}},\ldots,x_{m})\\
=d h^0(t)_{m}(x_{1},\ldots,x_{m}) + l_{m}(h^0(t)_{1}(x_{1}),\ldots,h^0(t)_{1}(x_{m}))
\end{multline} 
and
\begin{multline}
\label{eq:multi2}  
\sum_{1 \leq i < j \leq m} 
(-1)^{i+j+1}h^1(t)_{m-1}([x_{i},x_{j}],x_{1},\ldots,\widehat{x_{i}},\ldots,\widehat{x_{j}},\ldots,x_{m})\\
=d h^1(t)_{m}(x_{1},\ldots,x_{m})
+(-1)^{1-m}\frac{\partial}{\partial t}h^0(t)_{m}(x_{1},\ldots,x_{m}),
\end{multline} 
where we used $m\ge 2$ and degree counting both to replace $D$ by
the de Rham differential $d$, and to conclude that  $l_m$ vanishes if one of its arguments is of the form $h^1(t)_{1}(x_{i})$.

C) Eq. \eqref{eq:nol1}, \eqref{eq:multi1} and \eqref{eq:multi2} are equivalent to the fact that the following two equations hold for all $t\in \RR$:
\begin{align}\label{uno}
h^0(t)\colon \g &\rightsquigarrow L_{\infty}(M,\omega)\text{ is an $\li$- {morphism}}\\
\label{due}
d_{tot}\overline{h^1(t)}=&\frac{\partial}{\partial t}\overline{h^0(t)}, \end{align}
where the bar denotes the following: if $\xi=\sum_{k=1}^N\xi_k$ with $\xi_k\in \wedge^{k} \g^*\otimes \Omega(M)$, then $$\bar{\xi}:=\sum_{k=1}^N\varsigma(k) \xi_k.$$
The equivalence between eq. \eqref{eq:multi1} (for all $2 \leq m \leq n+1$) and eq. \eqref{uno} is given again by \cite[\S 3.2]{FRZ}. We show the equivalence between eq.  \eqref{eq:nol1} and eq. \eqref{eq:multi2} (for all $2 \leq m \leq n+1$) on one side, and eq. \eqref{due} on the other. 
Notice that the L.H.S. of eq. \eqref{due} consists of three kinds of terms, exactly as it happens in eq. \eqref{eq:third}, \eqref{eq:first},  \eqref{eq:second}. The  analogue of eq. \eqref{eq:third} is equivalent to
 eq. \eqref{eq:nol1}. To take care of the analogues of eq.  \eqref{eq:first} and \eqref{eq:second},
write eq. \eqref{eq:multi2} in the form $-d_{\g}h^1(t)_{m-1}=dh^1(t)_{m}+(-1)^{1-m}\frac{\partial}{\partial t}h^0(t)_{m}$ for all $2 \leq m \leq n+1$, 
and use that $h^1(t)_{n}=0$ by degree reasons.
\hfill$\bigtriangleup$\\


Now, given homotopy moment maps $f,f'$, let $\prim:=\bar{f},\prim':=\bar{f'}$ (where the bar has been defined just above).
 
 Assume first that $f$ and $f'$ are inner equivalent,
  so that there is  $\eta\in (\wedge^{\ge 1} \g^*\otimes \Omega(M))_{n-1}$ with $\prim'-\prim =\dw \eta$.
 Define
 $$H= h^0(t)+h^1(t)dt:=(\overline{\prim+td_{tot}\eta})+\overline{\eta} dt.$$ 
 Notice that   $H$ satisfies eq. \eqref{H10}. Now we check the two conditions appearing in the above claim.
Condition \eqref{uno} is satisfied, because $d_{tot}\overline{h^0(t)}=d_{tot}\prim=\widetilde{\omega}$ for all $t$ and because of Prop. \ref{lem:doublemomap}. Condition \eqref{due} is   satisfied  as both sides are equal to
 $ {d_{tot}\eta}$.   Therefore by the claim $H$ is an $\li$-morphism. We conclude that $f$ and $f'$ are equivalent the sense of Def. \ref{def:eqdolg}.
 
 Conversely, assume we are given $H=h^0(t)+h^1(t)dt$ satisfying the conditions of Def. \ref{def:eqdolg}, that is: $h^0(1)=f'=\overline{\prim'}$ and
$h^0(0)=f=\overline{\prim}$, and $H$ is an $\li$-morphism. Then integrating over $t$ we define
$$\eta:=\int_0^1\overline{h^1(t)}.$$ It satisfies 
$$\prim'-\prim=\overline{h^0(1)}-\overline{h^0(0)}=\int_0^1\frac{\partial}{\partial t}\overline{h^0(t)}=d_{tot}\eta,$$ where in the last equation
condition \eqref{due} is used.
 Hence $f$ and $f'$ are inner equivalent.
\end{proof}

\bibliographystyle{habbrv} 
\bibliography{CoHo}

\begin{thebibliography}{10}

\bibitem{AriasUribe}
C.~Arias~Abad and B.~Uribe.
\newblock {An $A_{\infty}$-version of the Chern-Weil homomorphism}.
\newblock {In progress}.

\bibitem{audin2004torus}
M.~Audin.
\newblock {\em Torus Actions On Symplectic Manifolds}.
\newblock Progress in Mathematics. Birkh{\"a}user Basel, 2004.

\bibitem{BT}
R.~Bott and L.~W. Tu.
\newblock {\em Differential forms in algebraic topology}, volume~82 of {\em
  Graduate Texts in Mathematics}.
\newblock Springer-Verlag, New York, 1982.

\bibitem{brylinski}
J.~Brylinski.
\newblock {\em Loop Spaces, Characteristic Classes and Geometric Quantization}.
\newblock Modern Birkh{\"a}user Classics. Birkh{\"a}user Boston, 2007.

\bibitem{Buijs}
U.~Buijs and A.~Murillo.
\newblock {Algebraic models of non-connected spaces and homotopy theory of
  $L_\infty$ algebras}.
\newblock 04 2012, 1204.4999v1.

\bibitem{FRZ}
M.~Callies, Y.~Fregier, C.~L. Rogers, and M.~Zambon.
\newblock {Homotopy moment maps}.
\newblock {ArXiv:1304.2051}.

\bibitem{Ana}
A.~Cannas~da Silva.
\newblock {\em Lectures on symplectic geometry}, volume 1764 of {\em Lecture
  Notes in Mathematics}.
\newblock Springer-Verlag, Berlin, 2001.

\bibitem{DolgErr}
V.~A. Dolgushev.
\newblock {Erratum to: "A Proof of Tsygan's Formality Conjecture for an
  Arbitrary Smooth Manifold"}.
\newblock math/0703113.

\bibitem{FiorenzaRogersUrsPrequantum}
D.~Fiorenza, C.~L. Rogers, and U.~Schreiber.
\newblock {Higher geometric prequantum theory}.
\newblock {\em {Homology Homotopy Appl.}}, 16(2):107--142, 2014.

\bibitem{ForgerPoisMulti}
M.~Forger, C.~Paufler, and H.~R{{\"o}}mer.
\newblock The {P}oisson bracket for {P}oisson forms in multisymplectic field
  theory.
\newblock {\em Rev. Math. Phys.}, 15(7):705--743, 2003.

\bibitem{GSSusy}
V.~W. Guillemin and S.~Sternberg.
\newblock {\em Supersymmetry and equivariant de {R}ham theory}.
\newblock Mathematics Past and Present. Springer-Verlag, Berlin, 1999.
\newblock With an appendix containing two reprints by Henri Cartan [ MR0042426
  (13,107e); MR0042427 (13,107f)].

\bibitem{LadaStasheff}
T.~Lada and J.~Stasheff.
\newblock Introduction to {SH} {L}ie algebras for physicists.
\newblock {\em Internat. J. Theoret. Phys.}, 32(7):1087--1103, 1993.

\bibitem{MadsenSwannClosed}
T.~B. Madsen and A.~Swann.
\newblock Closed forms and multi-moment maps.
\newblock {\em Geom. Dedicata}, 165:25--52, 2013.

\bibitem{MeinEqCoho}
E.~Meinrenken.
\newblock {Equivariant cohomology and the Cartan model}.
\newblock {\em {Encyclopedia of Mathematical Physics, Elsevier. Availabe at
  \texttt{http://www.math.toronto.edu/mein/research/pubnew.html}}}, 2006.

\bibitem{RogersL}
C.~L. Rogers.
\newblock {$L_\infty$}-algebras from multisymplectic geometry.
\newblock {\em Lett. Math. Phys.}, 100(1):29--50, 2012.

\bibitem{WurzRyvkinMomaps}
L.~Ryvkin and T.~Wurzbacher.
\newblock Existence and unicity of co-moments in multisymplectic geometry.
\newblock {\em Diff. Geom. and Appl.}, 41:1--11, 2015.
\newblock Also available as
  \href{http://arxiv.org/abs/1411.2287}{ArXiv:1411.2287}.

\bibitem{ProductMomaps}
C.~S. Shahbazi and M.~Zambon.
\newblock {Products of multisymplectic manifolds and homotopy moment maps}.
\newblock {ArXiv:1504.08194}.

\bibitem{HDirac}
M.~Zambon.
\newblock {$L_\infty$}-algebras and higher analogues of {D}irac structures and
  {C}ourant algebroids.
\newblock {\em J. Symplectic Geom.}, 10(4):563--599, 2012.

\end{thebibliography}
\end{document}